\documentclass{amsart}
\usepackage{tikz,cleveref}
\usetikzlibrary{cd}

\title{Direct limits of graded matrix algebras}
\author{Mikhail Kochetov}
\address{Department of Mathematics and Statistics, Memorial University of Newfoundland, St. John's, NL, A1C5S7, Canada.}
\email{mikhail@mun.ca}
\author{Felipe Yasumura}
\address{Department of Mathematics, Instituto de Matem\'atica, Estat\'istica e Ci\^encia da Computa\c c\~ao, Universidade de S\~ao Paulo, SP, Brazil}
\email{fyyasumura@ime.usp.br}

\thanks{M.K.~is supported by Discovery Grant 2018-04883 of the Natural Sciences and Engineering Research Council (NSERC) of Canada.}
\thanks{F.Y.~is supported by Fapesp grant 2024/14914-9.}

\newtheorem{Thm}{Theorem}
\newtheorem{lemma}[Thm]{Lemma}
\newtheorem{proposition}[Thm]{Proposition}
\newtheorem{corollary}[Thm]{Corollary}
\theoremstyle{definition}
\newtheorem{definition}[Thm]{Definition}
\theoremstyle{remark}
\newtheorem{remark}[Thm]{Remark}
\newtheorem{example}[Thm]{Example}

\newcommand{\under}[2]{\underset{#1}{#2}}
\renewcommand{\over}[2]{\overset{#1}{#2}}
\DeclareMathOperator{\supp}{supp}
\DeclareMathOperator{\Supp}{Supp}
\DeclareMathOperator{\Hom}{Hom}
\DeclareMathOperator{\End}{End}

\newcommand{\CC}{\mathbb{C}}
\newcommand{\FF}{\mathbb{F}}
\begin{document}
\begin{abstract}
The direct limit of finite-dimensional semisimple associative algebras is a purely algebraic counterpart to an important class of $C^\ast$-algebras. It is well known that, over an algebraically closed field $\mathbb{F}$, such limits are classified using their $K_0$ groups. More recently, Hazrat has developed a $G$-graded version of this theory for an abelian group $G$, introducing $\mathbb{Z}G$-modules $K_0^\mathrm{gr}$ and showing that they classify certain direct limits. In this paper, we consider direct limits of matrix algebras over $\mathbb{F}$ endowed with gradings by a finite abelian group $G$. In particular, we give an explicit description of $K_0^\mathrm{gr}$ modules of such limits, give conditions under which a direct limit of matrix algebras with elementary gradings ``absorbs'' a graded-division algebra, and show how this can be used to classify direct limits of matrix algebras with arbitrary $G$-gradings.
\end{abstract}
\maketitle

\section{Introduction}

Direct limits of finite-dimensional algebras have important connections to the theory of $C^\ast$-algebras in Functional Analysis (see, e.g., \cite{Bra,Dix}) and have also inspired purely algebraic investigations (e.g., \cite{BZ,Zal}).
The case of direct limits of complex semisimple associative algebras (i.e., finite direct products of matrix algebras over $\CC$) is well understood through their groups $K_0$, following the foundational works of Bratteli \cite{Bra} and Elliott \cite{E}. In the context of Functional Analysis, homomorphisms are required to respect the involution that is part of the definition of $C^\ast$-algebra, but it was already noted in \cite{E} that the involution does not play a role in the classification of direct limits, which is, in fact, the same over any algebraically closed field.

It is also natural to consider direct limits of algebras with some additional structure such as an action or grading by a group. In \cite{HaR}, direct limits of finite-dimensional $C^\ast$-algebras with an action of a compact group were classified in terms of $K_0$ groups under some technical restrictions on the action, which were later removed in \cite{Wang} at the expense of considering a more complicated invariant (a whole family of $K_0$ groups). A version of $K_0$ group was introduced by Hazrat in \cite{Haz} for an associative algebra $\mathcal{A}$ graded by an abelian group $G$ (see also \cite{HazVas} for graded algebras equipped with an involution that inverts degrees), and it is richer than in the ungraded case: $K_0^\mathrm{gr}(\mathcal{A})$ is not just an abelian group, but a $\mathbb{Z}G$-module.
It was shown in \cite[Ch.~5]{Haz} that this module (equipped with a partial ordering and an order-unit) classifies direct limits of finite direct products of matrix algebras with coefficients in a fixed graded field. Here we are interested in direct limits of simple finite-dimensional associative algebras equipped with a $G$-grading. 

Gradings by various groups on different kinds of finite-dimensional algebras (associative, Lie, Jordan, etc.) have been extensively studied, especially in the case of simple finite-dimensional algebras, see, e.g., the monograph \cite{EK13} and the references therein. In particular, gradings on matrix algebras $\mathrm{M}_n(\FF)$ over an algebraically closed field $\FF$ of characteristic $0$ were described in \cite{BaSeZa2001}. Even though specific gradings on infinite-dimensional algebras appear frequently (e.g., gradings by the weight lattice in the theory of Kac--Moody Lie algebras and their representations), the classification of gradings by arbitrary groups has not yet received much attention in the infinite-dimensional context. We can mention \cite{BaBrKoch}, which classifies $G$-gradings on the algebras of operators of finite rank --- such as the algebra of finitary matrices $\mathrm{M}_\infty(\FF)$, which was treated earlier in \cite{BaZa2010} --- and on the related finitary simple Lie algebras, and \cite{WY}, which classifies $G$-gradings on certain algebras of triangularizable operators.

We want to study $G$-gradings on associative algebras that are locally finite-di\-men\-sional simple, i.e., associative algebras in which every finitely generated subalgebra is contained in a finite-dimensional simple subalgebra. These gradings naturally arise as direct limits of gradings on finite-dimensional simple (associative) algebras. In this paper we will focus on direct limits of unital embeddings of such algebras, which is the case in a certain sense opposite to $M_\infty(\FF)$, see Baranov and Zhilinskii \cite{BZ}. By a graded version of Wedderburn Theorem, a finite-dimensional simple algebra $\mathcal{R}$ equipped with a $G$-grading is isomorphic to a matrix algebra over a graded-division algebra $\mathcal{D}$. Since $\mathcal{D}$ is not necessarily commutative and a homomorphism of graded algebras $\mathcal{R}\to\mathcal{R}'$ does not necessarily respect the corresponding graded-division algebras, the result of Hazrat mentioned above does not apply in general, but only to the so-called elementary gradings. Our purpose will then be two-fold: for a finite abelian group $G$, classify direct limits of such graded algebras up to isomorphism (see \Cref{last_prop}) and give an explicit description of the graded $K_0$ groups that appear in this classification. Such a description is possible because we consider a direct limit of \emph{simple} finite-dimensional algebras: in the ungraded case, the $K_0$ group of the limit is then an additive subgroup of $\mathbb{Q}$ (with its natural ordering), while in our case, we show that it is a $\mathbb{Z}G$-submodule of $\mathbb{Q}\overline{G}$ for a certain quotient $\overline{G}$ of $G$ (see \Cref{K_0_ultraelementary,K_0_division}). This description allows us to give a more explicit form of Hazrat's result in our setting (see \Cref{thm1}) and give a criterion to determine which direct limits can be obtained from elementary gradings (see \Cref{thm}).

\section{Preliminaries}
Let $\mathbb{F}$ be a field and $G$ be an abelian group. All algebras under consideration will be associative, unital, and over $\mathbb{F}$. We will later assume that $\mathbb{F}$ is algebraically closed and $G$ is finite, but 
most results in this section do not require these assumptions.

\subsection{Graded algebras and graded modules} We will now review some basic concepts concerning graded algebras, which can be found, for example, in Chapters 1 and 2 of the monograph \cite{EK13}. A $G$-grading $\Gamma$ on an algebra $\mathcal{A}$ is a decomposition into a direct sum of vector subspaces $\mathcal{A}=\bigoplus_{g\in G}\mathcal{A}_g$ such that $\mathcal{A}_g\mathcal{A}_h\subseteq\mathcal{A}_{gh}$, for all $g,h\in G$. A $G$-graded algebra is an algebra with a fixed $G$-grading, i.e., a pair $(\mathcal{A},\Gamma)$, which we will write as $\mathcal{A}$ when it does not cause confusion. The nonzero elements in each subspace $\mathcal{A}_g$ are called \emph{homogeneous}, and we write $\deg a=g$ for $a\in\mathcal{A}_g\setminus\{0\}$. We allow some of the $\mathcal{A}_g$ to be zero, so we define the \emph{support} of the grading, denoted by $\Supp\Gamma$ or $\Supp\mathcal{A}$, to be the set of all $g\in G$ such that $\mathcal{A}_g\ne0$. A subspace $\mathcal{B}$ of $\mathcal{A}$ is \emph{graded} if it is the direct sum of the intersections $\mathcal{B}\cap\mathcal{A}_g$, $g\in G$. In particular, we can speak of graded subalgebras and graded ideals. A graded algebra $\mathcal{A}$ (with nonzero product) is said to be \emph{graded-simple} if it has no graded ideals except $0$ and $\mathcal{A}$.

Given two $G$-graded algebras $\mathcal{A}=\bigoplus_{g\in G}\mathcal{A}_g$ and $\mathcal{B}=\bigoplus_{g\in G}\mathcal{B}_g$, a \emph{homomorphism of graded algebras} $\mathcal{A}\to\mathcal{B}$ is an algebra homomorphism $f:\mathcal{A}\to\mathcal{B}$ such that $f(\mathcal{A}_g)\subseteq\mathcal{B}_g$, for each $g\in G$. If $f$ is an isomorphism, then we say that $\mathcal{A}$ and $\mathcal{B}$ are \emph{isomorphic as graded algebras} (or \emph{$G$-graded-isomorphic}) and write $\mathcal{A}\cong\mathcal{B}$.

If $\alpha:G\to H$ is a group homomorphism, then any $G$-grading $\Gamma$ on an algebra $\mathcal{A}$ gives rise to an $H$-grading on $\mathcal{A}$ whose homogeneous components are defined by $\mathcal{A}_h:=\bigoplus_{g\in\alpha^{-1}(h)}\mathcal{A}_g$, for all $h\in H$. This grading is denoted by ${}^\alpha\Gamma$, and the $H$-graded algebra $(\mathcal{A},{}^\alpha\Gamma)$ by ${}^\alpha\mathcal{A}$.

A \emph{graded-division} algebra is a $G$-graded unital algebra in which every nonzero homogeneous element is invertible. A finite-dimensional graded-division algebra $\mathcal D$ over an algebraically closed field $\mathbb{F}$ is isomorphic to a twisted group algebra $\mathbb{F}^\sigma T$, with its natural $T$-grading regarded as a $G$-grading, where $T\leq G$ is the support of $\mathcal D$ and $\sigma:T\times T\to\mathbb{F}^\times$ is a 2-cocycle with respect to the trivial action of $T$ on the multiplicative group $\mathbb{F}^\times$ of the base field. Thus, the  isomorphism classes of these graded-division algebras are parametrized by the pairs $(T,[\sigma])$ where $[\sigma]$ is the cohomology class of $\sigma$ in $H^2(T,\mathbb{F}^\times)$. Furthermore, $\mathcal D$ is simple as an ungraded algebra if and only if $\sigma$ is nondegenerate; since we assume that $G$ is abelian, it is known that in this case $T$ is a direct product of squares of cyclic groups.

A $G$-graded left $\mathcal{A}$-module is a left $\mathcal{A}$-module $\mathcal{V}$ endowed with a decomposition $\mathcal{V}=\bigoplus_{g\in G}\mathcal{V}_g$ such that $\mathcal{A}_g\mathcal{V}_h\subseteq\mathcal{V}_{gh}$, for all $g$, $h\in G$. Given a $G$-graded left $\mathcal{A}$-module $\mathcal{V}$ and an element $g\in G$, the \emph{(right) shift} of $\mathcal{V}$ by $g$, denoted $\mathcal{V}^{[g]}$, is the left $\mathcal{A}$-module $\mathcal{V}$ endowed with the following $G$-grading: $\mathcal{V}^{[g]}_h:=\mathcal{V}_{hg^{-1}}$, for each $h\in G$. Note that $\mathcal{V}^{[g]}$ is a $G$-graded left $\mathcal{A}$-module as well. Graded right modules and their shifts are defined similarly. Any graded module over a graded-division algebra admits a basis consisting of homogeneous elements.

Given a $G$-graded algebra $\mathcal{A}$ and an $n$-tuple $\gamma=(g_1,\ldots,g_n)\in G^n$, we obtain a $G$-grading on $\mathrm{M}_n(\mathcal{A})$, by declaring
\begin{equation*}
\deg E_{ij}(a):= g_i(\deg a)g_j^{-1}\quad\text{for nonzero homogeneous }a\in\mathcal{A},
\end{equation*}
where $E_{ij}(a)$ denotes the matrix whose entry in position $(i,j)$ is $a$ and all other entries are $0$. In particular, for $\mathcal{A}=\mathbb{F}$, we obtain gradings on $\mathrm{M}_n(\mathbb{F})$, which are called \emph{elementary}. 

If we permute the entries of $\gamma$, we will obtain an isomorphic $G$-grading, so it is convenient to consider the multiset in $G$ corresponding to $\gamma$, i.e., the map $x_\gamma:G\to\mathbb{Z}_{\ge0}$ where
$$
x_\gamma(g)=|\{i\mid g_i=g\}|.
$$
The size of a multiset $x:G\to\mathbb{Z}_{\ge0}$ is $|x|:=\sum_{g\in G}x(g)$. We say that $x$ is finite if this sum is finite.
The finite multisets in $G$ can be interpreted as elements of the group semiring $\mathbb{Z}_{\ge0}G$, which will allow us to multiply them. More precisely,  a finite multiset $x$ is identified with the element $\sum_{g\in G}x(g)g\in\mathbb{Z}_{\ge 0}G$. For example, if $x(e)=n$ and $x(g)=0$ for all $g\ne e$, we obtain the element $ne$, which we will write as $n$, since $e$ is the identity element of the semiring $\mathbb{Z}_{\ge 0}G$. 

For any multiset $x$ in $G$ of size $n$, we can choose a corresponding $n$-tuple $\gamma$ and define a $G$-grading on $\mathrm{M}_n(\mathcal{A})$. 
Since all choices of $\gamma$ lead to isomorphic $G$-graded algebras, we will abuse notation and denote any of them by $\mathrm{M}_x(\mathcal{A})$. For example, $\mathrm{M}_n(\mathcal{A})$ in this notation would refer to the case where the grading comes entirely from $\mathcal{A}$, which gives the trivial grading if $\mathcal{A}=\mathbb{F}$. It is easy to check that, for any nonzero $x,y\in\mathbb{Z}_{\ge 0}G$, we have 
\[
\mathrm{M}_x(\mathrm{M}_y(\mathcal{A}))\cong\mathrm{M}_{xy}(\mathcal{A}).
\]
We should also note that, since we assume that $G$ is abelian, the tensor product of $G$-graded algebras $\mathcal A$ and $\mathcal B$ over $\mathbb{F}$ is $G$-graded by defining 
\[
\deg(a\otimes b):=(\deg a)(\deg b)\quad\text{for nonzero homogeneous }a\in\mathcal A\text{ and }b\in\mathcal B.
\]
Then we have $\mathrm{M}_x(\mathcal{A})\cong \mathrm{M}_x\otimes\mathcal{A}$, where $\mathrm{M}_x:=\mathrm{M}_x(\mathbb{F})$. It follows that
\[
\mathrm{M}_x(\mathcal{A})\otimes\mathrm{M}_y(\mathcal{B})\cong\mathrm{M}_{xy}(\mathcal{A}\otimes\mathcal{B}).
\]

Now, a graded version of Wedderburn Theorem says that any finite-dimensional graded-simple associative algebra is $G$-graded-isomorphic to the endomorphism algebra, $\End_\mathcal{D}(\mathcal{V})$, of a graded right module $\mathcal{V}$ over a graded-division algebra $\mathcal{D}$. The grading on $\End_\mathcal{D}(\mathcal{V})$ is defined by declaring a nonzero linear map $f:\mathcal{V}\to\mathcal{W}$ homogeneous of degree $g$ if $f(\mathcal{V}_h)\subseteq\mathcal{W}_{gh}$ for all $h\in G$. Fixing an ordered basis of $\mathcal{V}$ consisting of homogeneous elements, this algebra can be written as $\mathrm{M}_x(\mathcal{D})$, where $x$ is the multiset determined by the tuple of the degrees of the basis elements. 
Moreover, the graded-division algebra $\mathcal{D}$ is determined up to isomorphism, while the image of $x$ in $\mathbb{Z}_{\ge 0}(G/T)$ is determined up to shift, where $T=\Supp\mathcal{D}$. In contrast, for a direct limit of such graded matrix algebras, one cannot assign a unique graded-division algebra, as shown by the following example.

\begin{example}\label{ex1}
Let $\mathcal{Q}=\mathrm{M}_2(\mathbb{F})$ be endowed with the following grading by the Klein group $\mathbb{Z}_2\times\mathbb{Z}_2$. We assign the matrices
$$
X=\begin{pmatrix}1&0\\0&-1\end{pmatrix}\text{ and } Y=\begin{pmatrix}0&1\\1&0\end{pmatrix}
$$
degrees $(1,0)$ and $(0,1)$, respectively, and, since $X^2=I=Y^2$ and $XY=-YX$, this assignment defines a $\mathbb{Z}_2\times\mathbb{Z}_2$-grading on $\mathcal{Q}$ with components $\mathcal{Q}_{(i,j)}=\mathrm{Span}\{X^iY^j\}$, for each $(i,j)\in\mathbb{Z}_2\times\mathbb{Z}_2$. Clearly, this is a \emph{division grading} in the sense that it makes $\mathcal{Q}$ a graded-division algebra, with support $T=\mathbb{Z}_2\times\mathbb{Z}_2$. Then, we obtain a sequence of unital embeddings of the form $\mathcal{A}\to\mathcal{A}\otimes1\subseteq\mathcal{A}\otimes\mathcal{B}$ given by $a\mapsto a\otimes 1$:
$$
\mathcal{Q}\longrightarrow\mathcal{Q}\otimes\mathcal{Q}\longrightarrow\mathcal{Q}\otimes\mathcal{Q}\otimes\mathcal{Q}\longrightarrow\cdots
$$
Denote $\mathcal{Q}^{\otimes n}=\mathcal{Q}\otimes\cdots\otimes\mathcal{Q}$ (tensor product of $n$ copies of $\mathcal{Q}$). This sequence has two obvious subsequences:
$$
\begin{tikzcd}
\mathcal{Q}\arrow{rd}\arrow[dashed]{rr}&&\mathcal{Q}^{\otimes3}\arrow{rd}\arrow[dashed]{rr}&&\mathcal{Q}^{\otimes5}\arrow{rd}\arrow[dashed]{r}&\cdots\\%
&\mathcal{Q}^{\otimes2}\arrow[dashed]{rr}\arrow{ur}&&\mathcal{Q}^{\otimes4}\arrow[dashed]{rr}\arrow{ur}&&\cdots
\end{tikzcd}
$$
and the direct limit of the sequence is isomorphic to the direct limit of each of these two subsequences.
It is easy to see that $\mathcal{Q}^{\otimes2}\cong\mathrm{M}_x(\mathbb{F})$ where $x$ is the sum of the elements of $T$ in $\mathbb{Z}_{\geq 0}T$. Hence, we get $\mathcal{A}\otimes\mathcal{Q}\cong\mathcal{A}$, where $\mathcal{A}$ is the direct limit of
$$
\mathrm{M}_x(\mathbb{F})\longrightarrow\mathrm{M}_{x}(\mathbb{F})^{\otimes2}\longrightarrow\mathrm{M}_{x}(\mathbb{F})^{\otimes3}\longrightarrow\cdots.
$$
We will see in \Cref{thm} under what conditions the graded-division part will ``disappear'' in the limit.
\end{example}

\subsection{Graded Brauer group}

Fix an abelian group $G$. Similarly to the classical Brauer group, the \emph{$G$-graded Brauer group} of $\mathbb{F}$, which we denote $\mathrm{Br}_G(\mathbb{F})$, can be defined as the set of equivalence classes of (finite-dimensional) central simple $\mathbb{F}$-algebras, now equipped with a $G$-grading, and the group operation is induced by tensor product: $[\mathcal{A}][\mathcal{B}]:=[\mathcal{A}\otimes\mathcal{B}]$. The equivalence relation is defined as follows: $\mathcal{A}\sim\mathcal{B}$ if and only if $\mathrm{M}_x(\mathcal{A})\cong \mathrm{M}_y(\mathcal{B})$ for some $0\neq x,y\in\mathbb{Z}_{\geq 0}G$. In particular, the identity element, $[\mathbb{F}]$, consists of the matrix algebras over $\mathbb{F}$ equipped with an elementary grading. The existence of inverses is shown in the usual way: if a central simple algebra $\mathcal{A}$ is equipped with a $G$-grading, the same homogeneous components define a grading on the opposite algebra $\mathcal{A}^{\mathrm{op}}$, and we have $\mathcal{A}\otimes\mathcal{A}^{\mathrm{op}}\cong \End_\mathbb{F}(\mathcal{A})$ as graded algebras, hence $[\mathcal{A}^\mathrm{op}]$ is the inverse of $[\mathcal{A}]$ in $\mathrm{Br}_G(\mathbb{F})$. 

It follows from the graded Wedderburn Theorem that each equivalence class contains a graded-division algebra, which is unique up to isomorphism. Thus, $\mathrm{Br}_G(\mathbb{F})$ allows us to keep track of the division part when computing tensor products: if $\mathcal{A}$ and $\mathcal{A}'$ are central simple algebras endowed with a $G$-grading, then $\mathcal{A}\cong\mathrm{M}_x(\mathcal{D})$ and $\mathcal{A}'\cong\mathrm{M}_{x'}(\mathcal{D}')$, where $\mathcal{D}$ and $\mathcal{D}'$ are graded-division algebras, and we have 
$$
\mathcal{A}\otimes\mathcal{A}'=\mathrm{M}_{xx'y}(\mathcal{E}),
$$
where $\mathcal{E}$ is the graded-division algebra determined by $[\mathcal{E}]=[\mathcal{D}][\mathcal{D}']$ and $y\in\mathbb{Z}_{\ge 0}G$ satisfies $\mathcal{D}\otimes\mathcal{D}'\cong\mathrm{M}_y(\mathcal{E})$.
 
In general, the graded Brauer group $\mathrm{Br}_G(\mathbb{F})$ is complicated, since it contains the classical Brauer group $\mathrm{Br}(\mathbb{F})$. If $\mathbb{F}$ is algebraically closed and $G$ is finite, then $\mathrm{Br}_G(\mathbb{F})\cong\Hom(\hat{G}^{\wedge 2},\mathbb{F}^\times)$ (see e.g. \cite{EK2015}). Here $\hat{G}$ is the dual group of $G$, which can be defined as $\Hom(G,\mathbb{Q}/\mathbb{Z})\cong\Hom(G,\mathbb{C}^\times)$, and the wedge denotes the exterior power in the category of abelian groups (as $\mathbb{Z}$-modules). We also note that, since in this case the nonzero components of a graded-division algebra $\mathcal{D}$ are 1-dimensional, we have
\begin{equation}\label{DDop}
\mathcal{D}\otimes\mathcal{D}^\mathrm{op}\cong\mathrm{M}_{x_T}(\mathbb{F}),
\end{equation}
where $T=\Supp\mathcal{D}$ and $x_T:=\sum_{t\in T}t\in\mathbb{Z}_{\ge0}G$.

\subsection{\label{subsec:graded_K0}Graded $K_0$ groups}
We will now recall the basics of the graded version of $K_0$ groups, denoted $K_0^\mathrm{gr}$ or $K_0^G$ when it is necessary to indicate the grading (abelian) group $G$, following the book by Hazrat \cite{Haz}. Let $\mathcal{A}$ be a unital $G$-graded algebra. A \emph{graded-projective} module is a projective object in the category of graded modules. It is known that a graded-projective module is the same as a projective module with a grading (see, e.g., \cite[Proposition 1.2.15]{Haz}). 

We consider the set of isomorphism classes of finitely generated, projective and graded (left) $\mathcal{A}$-modules, denoted by $K_0^\mathrm{gr}(\mathcal{A})^+$. Given $[P]$, $[Q]\in K_0^\mathrm{gr}(\mathcal{A})^+$, we set $[P]+[Q]:=[P\oplus Q]$. This operation is well defined and makes $K_0^\mathrm{gr}(\mathcal{A})^+$ a commutative monoid. Denote by $K_0^\mathrm{gr}(\mathcal{A})$ its group completion (also known as its Grothendieck group), i.e., the set of all equivalence classes of $[P]-[Q]$ for $[P]$, $[Q]\in K_0^\mathrm{gr}(\mathcal{A})^+$ (see, e.g., \cite[Section 3.1.1]{Haz}). Then, $(K_0^\mathrm{gr}(\mathcal{A}),K_0^\mathrm{gr}(\mathcal{A})^+)$ is a partially ordered group, where, for $x$, $y\in K_0^\mathrm{gr}(\mathcal{A})$, one defines $x\ge y$ if and only if $x-y\in K_0^\mathrm{gr}(\mathcal{A})^+$ (see, e.g., \cite[Example 3.6.3]{Haz}).

The new feature, compared to the ordinary $K_0$ group, is that $K_0^\mathrm{gr}(\mathcal{A})$ has a structure of $\mathbb{Z}G$-module, by letting each $g\in G$ act as a shift: $g[P]:=[P^{[g]}]$. Moreover, $K_0^\mathrm{gr}(\mathcal{A})^+$ is invariant under the action of $\mathbb{Z}_{\ge0}G$, so $K_0^\mathrm{gr}(\mathcal{A})$ is a partially ordered $\mathbb{Z}G$-module. 

The equivalence class $[\mathcal{A}]$ of the free module of rank $1$ is a distinguished element in $K_0^\mathrm{gr}(\mathcal{A})$, and it is called the \emph{order-unit}. 
%It satisfies the following property: for any $[P]\in K_0^\mathrm{gr}(\mathcal{A})^+$, there exists $x\in\mathbb{Z}_{\ge0}G$ such that $[P]\le x[\mathcal{A}]$. 

We include the partial order, $\mathbb{Z}G$-action, and the order-unit in the definition of $K_0^\mathrm{gr}(\mathcal{A})$, which will stand for the triple $(K_0^\mathrm{gr}(\mathcal{A}),K_0^\mathrm{gr}(\mathcal{A})^+,[\mathcal{A}])$. A morphism from $(K,K^+,x_0)$ to $(K',K^{\prime+},x_0')$ means a $\mathbb{Z}G$-module homomorphism $\varphi:K\to K'$ such that $\varphi(K^+)\subseteq K^{\prime+}$ and $\varphi(x_0)=x_0'$. Then, $K_0^\mathrm{gr}$ becomes a functor, and it is known to commute with direct limits \cite[Theorem 3.2.4]{Haz}. 
%This group can also be defined in terms of homogeneous idempotents (see, for instance, \cite[Section 3.2]{Haz}), which can be used to define $K_0^\mathrm{gr}$ of non-unital algebras.

It turns out that this functor is sufficient to classify certain direct limits of ``graded matricial algebras'', i.e., $G$-graded algebras of the form
$$
\mathrm{M}_{x_1}(\mathcal{C})\times\cdots\times\mathrm{M}_{x_t}(\mathcal{C}),
$$
where $x_1$, \dots, $x_t\in\mathbb{Z}_{\ge0}G$ and $\mathcal{C}$ is a fixed commutative graded-division algebra. A \emph{$G$-graded ultramatricial} $\mathcal{C}$-algebra is a direct limit of unital $\mathcal{C}$-linear embeddings of matricial $\mathcal{C}$-algebras (see \cite[Definition 5.2.1]{Haz}). Then, by \cite[Theorem 5.2.4]{Haz}, two $G$-graded ultramatricial $\mathcal{C}$-algebras $\mathcal{A}$ and $\mathcal{A}'$ are isomorphic if and only if the triples $(K_0^\mathrm{gr}(\mathcal{A}),K_0^\mathrm{gr}(\mathcal{A})^+,[\mathcal{A}])$ and $(K_0^\mathrm{gr}(\mathcal{A}'),K_0^\mathrm{gr}(\mathcal{A}')^+,[\mathcal{A}'])$ are isomorphic. This can be considered as the graded version of Elliot's result. 
%Moreover, they prove in \cite[Theorem 5.2.5]{Haz} that $\mathcal{A}$ and $\mathcal{A}'$ are Morita equivalent if and only if $(K_0^\mathrm{gr}(\mathcal{A}),K_0^\mathrm{gr}(\mathcal{A})^+)\cong(K_0^\mathrm{gr}(\mathcal{A}'),K_0^\mathrm{gr}(\mathcal{A}')^+)$ (i.e., isomorphic as ordered groups and not necessarily preserving the order-unit).

\subsection{Preliminaries on direct limits of graded matrix algebras}
We start with the following graded algebra version of a well-known isomorphism criterion for direct limits (see, for instance, \cite[Proposition 3.1]{BZ}):
\begin{lemma}\label{lem1}
Consider direct systems $(\mathcal{A}_j,\iota_{ji})$ and $(\mathcal{A}_j',\iota_{ji}')$ of finitely generated (graded) algebras where $j\in\mathbb{N}$, and $\iota_{ji}:\mathcal{A}_i\to\mathcal{A}_j$ and $\iota_{ji}':\mathcal{A}_i'\to\mathcal{A}_j'$ are embeddings. Then, $\lim\limits_{\longrightarrow}(\mathcal{A}_j,\iota_{ji})\cong\lim\limits_{\longrightarrow}(\mathcal{A}_j',\iota_{ji})$ if and only if there exist subsequences $\{j_k\}$, $\{j_k'\}$ and embeddings $\varepsilon_k:\mathcal{A}_{j_k}\to\mathcal{A}_{j_k'}'$, $\varepsilon_k':\mathcal{A}_{j_k'}'\to\mathcal{A}_{j_{k+1}}$ such that

\[
\begin{tikzcd}
\mathcal{A}_{j_1}\arrow{d}{\varepsilon_1}\arrow{r}{\iota_{j_2,j_1}}&\mathcal{A}_{j_2}\arrow{d}{\varepsilon_2}\arrow{r}{\iota_{j_3,j_2}}&\mathcal{A}_{j_3}\arrow{d}{\varepsilon_3}\arrow{r}{\iota_{j_4,j_3}}&\cdots\\%
\mathcal{A}_{j_1'}'\arrow{r}{\iota_{j_2',j_1'}'}\arrow{ur}{\varepsilon_1'}&\mathcal{A}_{j_2'}'\arrow{r}{\iota_{j_3',j_2'}'}\arrow{ur}{\varepsilon_2'}&\mathcal{A}_{j_3'}'\arrow{r}{\iota_{j_4',j_3'}'}\arrow{ur}{\varepsilon_3'}&\cdots
\end{tikzcd}
\]
\end{lemma}
% \begin{remark}
% The previous lemma works in a more general context: homomorphisms (not necessary embeddings) and a direct limit of finitely presented algebras
% \end{remark}

Two embeddings $\iota,\iota':\mathcal{A}\to\mathcal{B}$ of $G$-graded algebras are said to be \emph{equivalent} if $\iota'=\mathrm{Int}(b)\iota$, where $b\in\mathcal{B}$ is a homogeneous invertible element and $\mathrm{Int}(b):x\mapsto bxb^{-1}$ is the inner automorphism of $\mathcal{B}$ defined by $b$. (Since we are assuming that $G$ is abelian, this inner automorphism preserves degrees.)
\begin{lemma}\label{lm:equiv_embeddings}
Let $(\mathcal{A}_j,\iota_{ji})$ be a direct system of graded associative unital algebras where each $\iota_{ji}$ is a unital embedding. Given $(\mathcal{A}_j,\iota_{ji}')$, where $\iota_{ji}'$ is equivalent to $\iota_{ji}$, $\forall i,j\in\mathbb{N}$, then $\lim\limits_{\longrightarrow}(\mathcal{A},\iota_{ji})\cong\lim\limits_{\longrightarrow}(\mathcal{A},\iota_{ji}')$.
\end{lemma}
\begin{proof}
Consider the homogeneous invertible elements $a_i\in\mathcal{A}_i$, $i\ge 2$, such that $\iota_{i,i-1}'=\mathrm{Int}(a_i)\iota_{i,i-1}$. We define automorphisms $\varphi_k:\mathcal{A}_k\to\mathcal{A}_k$ recursively: $\varphi_1=\mathrm{id}$ and $\varphi_k=\mathrm{Int}(\prod_{i=2}^k\iota_{k,i}(a_{i}))$, for $k>1$. Then, $\iota_{k+1,k}'\varphi_k=\varphi_{k+1}\iota_{k+1,k}$, $\forall k$. This readily implies $\lim\limits_{\longrightarrow}(\mathcal{A},\iota_{ji})\cong\lim\limits_{\longrightarrow}(\mathcal{A},\iota_{ji}')$.
\end{proof}

%\subsection{Graded Noether--Skolem Theorem}

Recall that the Noether--Skolem Theorem says that every automorphism of a central simple algebra is inner. In the graded setting, we have \cite[Theorem 2.1]{RE}: if $G$ is an abelian group and an automorphism $\psi$ of a $G$-graded-simple algebra $\mathcal{B}$ is inner, then $\psi=\mathrm{Int}(b)$ for some homogeneous invertible $b\in\mathcal{B}$. This implies a graded version of the Noether--Skolem Theorem: if $\mathcal{B}$ is central simple and $G$-graded, then any automorphism of $\mathcal{B}$ as a graded algebra has the form $\mathrm{Int}(b)$ for $b$ as above.

A more general form of Noether--Skolem Theorem says that all unital embeddings of central simple algebras $\mathcal{A}\to\mathcal{B}$ are conjugate by an invertible element of $\mathcal{B}$. This does not hold in the graded setting. In fact, we have

\begin{proposition}\label{gr_NS}
Suppose that two central simple algebras, $\mathcal{A}$ and $\mathcal{B}$, are graded by an abelian group $G$. Let $\iota\colon\mathcal{A}\to\mathcal{B}$ and $\iota'\colon\mathcal{A}\to\mathcal{B}$ be unital embeddings of graded algebras. Then, $\iota$ and $\iota'$ are equivalent if and only if $\mathrm{Cent}_\mathcal{B}(\iota(\mathcal{A}))\cong\mathrm{Cent}_\mathcal{B}(\iota'(\mathcal{A}))$ as graded algebras.
\end{proposition}
\begin{proof}
Let $\mathcal{C}$ and $\mathcal{C}'$ be the above centralizers, which are graded subalgebras of $\mathcal{B}$ (since $G$ is abelian). If an automorphism $\psi$ of the graded algebra $\mathcal{B}$ satisfies $\iota'=\psi\iota$, then it maps $\iota(\mathcal{A})$ onto $\iota'(\mathcal{A})$ and, therefore, $\mathcal{C}$ onto $\mathcal{C}'$, so $\mathcal{C}\cong\mathcal{C}'$. Conversely, suppose $\theta\colon\mathcal{C}\to\mathcal{C}'$ is an isomorphism of $G$-graded algebras. 
By the Double Centralizer Theorem, the multiplication of $\mathcal{B}$ gives an isomorphism $\tau\colon\iota(\mathcal{A})\otimes\mathcal{C}\to\mathcal{B}$. 
Clearly, this is an isomorphism of $G$-graded algebras. Similarly, we get $\tau'\colon\iota'(\mathcal{A})\otimes\mathcal{C}'\to\mathcal{B}$. 
Let $\eta\colon\iota(\mathcal{A})\to\iota'(\mathcal{A})$ be the unique isomorphism such that $\iota'=\eta\iota$.
Then $\psi:=\tau'(\eta\otimes\theta)\tau^{-1}$ is an automorphism of $\mathcal{B}$ as a graded algebra and satisfies $\iota'=\psi\iota$.
Finally, by the graded version of the Noether--Skolem Theorem, we have $\psi=\mathrm{Int}(b)$ for some homogeneous invertible element $b\in\mathcal{B}$. 
\end{proof}

\begin{definition}\label{df:lim_as_prod}
Let $\mathcal{A}_i$ be central simple algebras endowed with $G$-gradings. A direct limit of unital embeddings,
\[
\mathcal{A}_{0}\over{\iota_1}{\longrightarrow}\mathcal{A}_{1}\over{\iota_2}{\longrightarrow}\mathcal{A}_{2}\over{\iota_3}{\longrightarrow}\cdots,
\]
will be written $\mathcal{A}_{0}\otimes\bigotimes_{i=1}^\infty\mathcal{C}_i$, where $\mathcal{C}_i\cong\mathrm{Cent}_{\mathcal{A}_i}(\iota_{i-1}(\mathcal{A}_{i-1}))$.
(In view of \Cref{lm:equiv_embeddings} and \Cref{gr_NS}, the limit depends only on the graded-isomorphism classes of $\mathcal{C}_i$.)
\end{definition}

In particular, for a unital embedding $\iota:\mathrm{M}_x(\mathbb{F})\to\mathrm{M}_y(\mathbb{F})$, the centralizer $\mathcal{C}=\mathrm{Cent}_{\mathrm{M}_y(\mathbb{F})}(\iota(\mathrm{M}_x(\mathbb{F})))$ satisfies $\mathrm{M}_y(\mathbb{F})\cong\mathrm{M}_x(\mathbb{F})\otimes\mathcal{C}$ and, hence, represents the identity element of $\mathrm{Br}_G(\mathbb{F})$: $\mathcal{C}\cong\mathrm{M}_a(\mathbb{F})$ for some element $a\in\mathbb{Z}_{\ge0}G$, which will be called a \emph{label} of $\iota$. It is determined up to shift (i.e., multiplication by an element of $G$) and may be chosen to satisfy $xa=y$. Conversely, any element $a\in\mathbb{Z}_{\ge0}G$ satisfying $xa=y$ determines an embedding $\mathrm{M}_x(\mathbb{F})\to\mathrm{M}_y(\mathbb{F})$ up to equivalence. 
%In other words, $\mathrm{Aut}^G(\mathrm{M}_y(\mathbb{F}))$-orbits of unital homomorphisms $\mathrm{M}_x(\mathbb{F})\to\mathrm{M}_y(\mathbb{F})$ are in bijection with $G$-orbits of $a\in\mathbb{Z}_{\ge0}G$ such that $xa=y$.

It turns out that the functor $K_0^\mathrm{gr}$ captures precisely the labels of embeddings.
To see this, we need the following realization of $K_0^\mathrm{gr}$ for graded matrix algebras.
\begin{lemma}\label{K_0_simple}
Let $G$ be an abelian group and $0\neq x=\sum_{g\in G}x_gg\in\mathbb{Z}_{\ge0}G$. Then,
\[
K_0^\mathrm{gr}(\mathrm{M}_x(\mathbb{F}))\cong(\mathbb{Z}G,\mathbb{Z}_{\ge0}G,\bar{x}),
\]
where $\bar{x}:=\sum_{g\in G}x_gg^{-1}$. This isomorphism is unique up to shift.
\end{lemma}
\begin{proof}
Let $n=|x|$ and pick an $n$-tuple $(g_1,\ldots,g_n)$ corresponding to $x$. If we define a $G$-grading on $\mathbb{F}^n$ by assigning degree $g_i$ to the $i$-th standard basis element, we get a graded $\mathrm{M}_x(\mathbb{F})$-module, which we denote by $\mathcal{V}=\mathcal{V}_x$.
%$\mathrm{M}_x(\mathbb{F})\cong\End_\mathbb{F}(\mathcal{V})$, so 
In fact, this is the unique graded-simple $\mathrm{M}_x(\mathbb{F})$-module up to isomorphism and shift. Hence, any finitely generated graded $\mathrm{M}_x(\mathbb{F})$-module is projective and a finite direct sum of copies of shifts of $\mathcal{V}$. Moreover, $\mathcal{V}^{[g]}\cong\mathcal{V}^{[h]}$ if and only if $g=h$. Thus, there is an isomorphism of ordered $\mathbb{Z}G$-modules $\mathbb{Z}G\to K_0^\mathrm{gr}(\mathrm{M}_x(\mathbb{F}))$ that sends $1\mapsto \mathcal{V}$. The distinguished element corresponds to the regular representation $\mathrm{M}_x(\mathbb{F})$, which is
$$
[\mathrm{M}_x(\mathbb{F})]=[\mathcal{V}\otimes\mathcal{V}^\ast]
=[\bigoplus_{g\in G}x_g\mathcal{V}^{[g^{-1}]}]=\bar{x}[\mathcal{V}],
$$
where $\mathcal{V}^\ast=\Hom_\mathbb{F}(\mathcal{V},\mathbb{F})$. Hence, $[\mathrm{M}_x(\mathbb{F})]$ corresponds to $\bar{x}$.
\end{proof}

\begin{lemma}\label{K_0_morphism}
Let $\mathcal{A}=\mathrm{M}_x(\mathbb{F})$ and $\mathcal{B}=\mathrm{M}_y(\mathbb{F})$. If $\iota:\mathcal{A}\to\mathcal{B}$ is a unital embedding with label $a$, then the morphism of triples $K_0^\mathrm{gr}(\iota):K_0^\mathrm{gr}(\mathcal{A})\to K_0^\mathrm{gr}(\mathcal{B})$ corresponds to the multiplication by $\bar{a}$ in the realization of \Cref{K_0_simple}.
\end{lemma}
\begin{proof}
Using the notation of \Cref{K_0_simple}, we have $K_0^\mathrm{gr}(\mathcal{A})\cong(\mathbb{Z}G,\mathbb{Z}_{\ge0}G,\bar{x})$ and $K_0^\mathrm{gr}(\mathcal{B})\cong(\mathbb{Z}G,\mathbb{Z}_{\ge0}G,\bar{y})$. Now, since $\mathcal{B}\cong \mathrm{M}_a(\mathbb{F})\otimes\mathcal{A}$, we obtain
$$
K_0^\mathrm{gr}(\iota)([\mathcal{V}_x]):=[\mathcal B\otimes_\mathcal{A}\mathcal{V}_x]=[\mathrm{M}_a(\mathbb{F})\otimes\mathcal{V}_x]=[\mathcal{V}_a\otimes\mathcal{V}_a^\ast\otimes\mathcal{V}_x]=\bar{a}[\mathcal{V}_a\otimes\mathcal{V}_x]=\bar{a}[\mathcal{V}_y].
$$
In the realization of \Cref{K_0_simple}, this means $K_0^\mathrm{gr}(\iota)(1)=\bar{a}$.
\end{proof}
We shall use the following notation for a label of an embedding: $\mathcal{A}\under{a}{\to}\mathcal{B}$, which induces $K_0^\mathrm{gr}(\mathcal{A})\over{\bar{a}}{\to}K_0^\mathrm{gr}(\mathcal{B})$.

\begin{lemma}\label{remove_matrix}
Let $n\in\mathbb{N}$ and let $\mathcal{A}$ and $\mathcal{A}'$ be limits of sequences of matrix algebras endowed with $G$-gradings. Then $\mathrm{M}_n\otimes\mathcal{A}\cong\mathrm{M}_n\otimes\mathcal{A}'$ if and only if $\mathcal{A}\cong\mathcal{A}'$.
\end{lemma}
\begin{proof}
The ``if'' direction is clear. For the ``only if'' direction, we consider an embedding $\eta:\mathrm{M}_{n}\otimes\mathrm{M}_x(\mathcal{D})\to\mathrm{M}_n\otimes\mathrm{M}_y(\mathcal{D}')$, where $\mathcal{D}$ and $\mathcal{D}'$ are graded-division algebras. Let $\mathcal{C}$ be the centralizer of the image of $\eta$.
%Then, the Double Centralizer Theorem gives $\mathrm{M}_n\otimes\mathrm{M}_{y}(\mathcal{D}')\cong\eta(\mathrm{M}_n\otimes\mathrm{M}_{x}(\mathcal{D}))\otimes\mathcal{C}$, where $\mathcal{C}=\mathrm{Cent}_{\mathrm{M}_{n}\otimes\mathrm{M}_{y}(\mathcal{D}')}(\eta(\mathrm{M}_{n}\otimes\mathrm{M}_{x}(\mathcal{D})))$. 
By the graded Wedderburn Theorem, $\mathcal{C}\cong\mathrm{M}_a(\mathcal{E})$ for some $a\in\mathbb{Z}_{\geq 0}G$ and a graded-division algebra $\mathcal{E}$. Moreover, the element $[\mathcal{E}]\in\mathrm{Br}_G(\mathbb{F})$ must satisfy $[\mathcal{D}']=[\mathcal{D}][\mathcal{E}]$, i.e., $\mathcal{D}\otimes\mathcal{E}\cong\mathrm{M}_z(\mathcal{D}')$. It follows that $\mathrm{M}_{ny}(\mathcal{D}')\cong\mathrm{M}_{nxaz}(\mathcal{D}')$, so replacing $a$ by its shift if necessary, we may assume that $ny\equiv nxaz$ modulo $T':=\Supp\mathcal{D}'$. Since $n\in\mathbb{N}$ is not a zero divisor in $\mathbb{Z}(G/T')$, this gives $y\equiv xaz$ modulo $T'$ and, hence, $\mathrm{M}_y(\mathcal{D}')\cong\mathrm{M}_x(\mathcal{D})\otimes\mathcal{C}$. Let $\varepsilon:\mathrm{M}_x(\mathcal{D})\to\mathrm{M}_y(\mathcal{D}')$ be the corresponding embedding. Then, by \Cref{gr_NS}, the embedding $\eta$ is equivalent to $\mathrm{id}_{\mathrm{M}_n}\otimes\varepsilon$. 
 
Now, let $\mathcal{A}=\lim\limits_{\longrightarrow}(\mathcal{A}_j,\iota_{ji})$ and $\mathcal{A}'=\lim\limits_{\longrightarrow}(\mathcal{A}_j',\iota_{ji}')$, and suppose $\mathrm{M}_n\otimes\mathcal{A}\cong\mathrm{M}_n\otimes\mathcal{A}'$. Then, by \Cref{lem1}, we can find a sequence of embeddings
$$
\mathrm{M}_n\otimes\mathcal{A}_{j_1}\stackrel{\eta_1}{\longrightarrow}\mathrm{M}_n\otimes\mathcal{A}'_{j_1'}\stackrel{\eta_1'}{\longrightarrow}\mathrm{M}_n\otimes\mathcal{A}_{j_2}\stackrel{\eta_2}{\longrightarrow}\cdots
$$
such that $\eta_k'\eta_{k}=\mathrm{id}\otimes\iota_{j_{k+1},j_k}$ and $\eta_{k+1}\eta_k'=\mathrm{id}\otimes\iota_{j_{k+1}',j_k'}'$. The first paragraph shows that, for some homogeneous invertible element $b'_1\in\mathrm{M}_n\otimes\mathcal{A}'_{j_1'}$ and an embedding $\varepsilon_1:\mathcal{A}_{j_1}\to\mathcal{A}'_{j'_1}$, we have $\mathrm{id}\otimes\varepsilon_1=\mathrm{Int}(b'_1)\eta_1$.
Then, we define $b_k'=(\mathrm{id}\otimes\iota_{j_k',j_1'}')(b'_1)$, $b_{k+1}=\eta_k'(b'_k)^{-1}$, $\mu_k'=\mathrm{Int}(b_{k+1})\eta_k'$ and $\mu_k=\mathrm{Int}(b_k')\eta_k$. By construction, we have $\mu_1=\mathrm{id}\otimes\varepsilon_1$, $\mu_k'\mu_{k}=\mathrm{id}\otimes\iota_{j_{k+1},j_k}$ and $\mu_{k+1}\mu_k'=\mathrm{id}\otimes\iota_{j_{k+1}',j_k'}'$. It follows by induction that $\mu_k|_{\mathrm{M}_n}=\mathrm{id}$ and $\mu'_k|_{\mathrm{M}_n}=\mathrm{id}$, i.e., $\mu_k=\mathrm{id}\otimes\varepsilon_k$ and $\mu_k'=\mathrm{id}\otimes\varepsilon_k'$ for some embeddings $\varepsilon_k:\mathcal{A}_{j_k}\to\mathcal{A}_{j_k'}'$ and $\varepsilon_k':\mathcal{A}_{j_k'}'\to\mathcal{A}_{j_{k+1}}$. It follows that $\varepsilon_k'\varepsilon_k=\iota_{j_{k+1},j_k}$ and $\varepsilon_{k+1}\varepsilon_k'=\iota_{j_{k+1},j_k}'$, so  
\Cref{lem1} tells us that $\mathcal{A}\cong\mathcal{A}'$.
\end{proof}

\section{Direct limits of matrix algebras with elementary gradings}

In this section, we consider direct limits of matrix algebras endowed with elementary $G$-gradings, for a finite abelian group $G$, i.e., $G$-graded algebras of the form $\mathcal{A}\cong\lim\limits_{\longrightarrow}\mathcal{A}_i$ where $\mathcal{A}_i=\mathrm{M}_{x_i}$. We may also write this limit as $\mathrm{M}_{x_0}\otimes\bigotimes_{i=1}^\infty\mathrm{M}_{a_i}$ (see \Cref{df:lim_as_prod}),  where $a_i$ is a label of $\iota_i:\mathcal{A}_{i-1}\to\mathcal{A}_i$. The goal is to obtain an explicit realization of $K_0^\mathrm{gr}(A)$ inside the $\mathbb{Z}G$-module $\mathbb{Q}G$, which is useful to determine when two such direct limits are isomorphic (see \Cref{thm1}). Moreover, in the next section, this realization will help us see when a graded-division algebra is ``absorbed'' by $\mathcal{A}$ (as illustrated in \Cref{ex1}).

Since $G$ is assumed to be finite, the group ring $\mathbb{Q}G$ is a direct product of fields. More precisely, $\mathbb{Q}G\cong\prod_{j}\mathbb{L}_j$, where each $\mathbb{L}_j$ is a (cyclotomic) finite extension of $\mathbb{Q}$ and $j$ runs over the orbit set $\hat{G}/\Gamma$ defined by the action of the absolute Galois group $\Gamma=\mathrm{Gal}(\overline{\mathbb{Q}}/\mathbb{Q})$ on $\hat{G}\cong\Hom(G,\overline{\mathbb{Q}}^\times)$ by post-composition. Indeed, $\mathbb{Q}G$ can be identified with the subalgebra of fixed points in the algebra of all maps $\hat{G}\to\overline{\mathbb{Q}}$ under the usual action of $\Gamma$, given by $(\sigma\cdot f)(\chi):=\sigma(f(\sigma^{-1}\cdot\chi))$ for all $\sigma\in\Gamma$, $f:\hat{G}\to\overline{\mathbb{Q}}$, and $\chi\in\hat{G}$. 

We denote by $e_j\in\mathbb{Q}G$ and $\pi_j:\mathbb{Q}G\to\mathbb{L}_j$ the (primitive) idempotents and projections corresponding to this direct product decomposition. Explicitly, $\pi_j(z)=e_jz$, for all $z\in\mathbb{Q}G$, and $e_j=\sum_{\chi\in j}e_\chi$ where $e_\chi=\frac{1}{|G|}\sum_{g\in G}\chi(g)^{-1}g$.  Also, for any subset $S\subseteq\hat{G}/\Gamma$, we have an ideal $\mathbb{L}_S:=\prod_{j\in S}\mathbb{L}_j$ of $\mathbb{Q}G$ and the corresponding projection $\pi_S:\mathbb{Q}G\to\mathbb{L}_S$.

For any $z\in\mathbb{Q}G$, we let $\supp(z):=\{j\in\hat{G}/\Gamma\mid\pi_j(z)\ne0\}$.

\begin{remark}\label{rem:conjugate}
This $\supp(z)$ should not be confused with the usual concept of support of an element $z=\sum_{g\in G} z_g g$ of the group ring $\mathbb{Q}G$, which is the set $\{g\in G\mid z_g\ne 0\}$.
It is also worth noting that $\supp(\bar{z})=\supp(z)$ where bar is the automorphism of $\mathbb{Q}G$ introduced in \Cref{K_0_simple}: $\bar{z}=\sum_{g\in G} z_g g^{-1}$. Indeed, on $e_j$ this automorphism coincides with the action of complex conjugation (an element of $\Gamma$), so $\bar{e}_j=e_j$.
\end{remark}

\begin{definition}\label{Def:standard_form}
We say that a limit $\mathrm{M}_{x_0}\otimes\bigotimes_{i=1}^\infty\mathrm{M}_{a_i}$ is in \emph{standard form} if $\supp a_i$ is a fixed set $S$, independent of $i$, and $\supp x_0\subseteq S$.
\end{definition}

\begin{proposition}\label{standard_form}
Every direct limit of matrix algebras endowed with elementary gradings is isomorphic to a limit in standard form.
\end{proposition}
\begin{proof}
Consider $\mathrm{M}_{x_0}\otimes\bigotimes_{i=1}^\infty\mathrm{M}_{a_i}$, i.e., the direct limit of the following sequence:
$$
\mathrm{M}_{x_0}\under{a_1}{\longrightarrow}\mathrm{M}_{x_{1}}\under{a_2}{\longrightarrow}\mathrm{M}_{x_{2}}\under{a_3}{\longrightarrow}\cdots,
$$
where $x_i=x_0\prod_{j=1}^ia_j$.
Let $S_0=\bigcap_{i=0}^\infty\supp(x_i)$, $S_n=\bigcap_{i=n}^{\infty}\supp(a_i)$, for $n\ge1$, and $S=\bigcup_{n=1}^\infty S_n$. Since  $\supp x_i=\supp(x_0)\cap\bigcap_{j=1}^i\supp(a_j)$, we have 
\[
S_0=\supp(x_0)\cap S_1\subseteq S_1\subseteq S_2\subseteq\cdots\subseteq S.
\]
These are finite sets, so there exists $N$ such that $S_n=S$, for all $n\ge N$. Skipping the terms before $x_{N-1}$ in the sequence, we may replace $x_0$ with $x'_0:=x_{N-1}$ and assume that $S_1=S$. Then, for each $n$, there exists $m>n$ such that $\supp\prod_{i=n+1}^ma_i=S$, so we may pass to a subsequence
$$
\mathrm{M}_{x_{0}'}\under{a_1'}{\longrightarrow}\mathrm{M}_{x_{i_1}}\under{a_2'}{\longrightarrow}\mathrm{M}_{x_{i_2}}\under{a_3'}{\longrightarrow}\cdots
$$
such that $\supp(a_i')=S$ $\forall i\ge1$.
\end{proof}
We will now see that the sets $S$ and $S_0$ introduced above depend only on the isomorphism class of the direct limit.
\begin{proposition}\label{S_invariant}
If $\mathrm{M}_{x_0}\otimes\bigotimes_{i=1}^\infty\mathrm{M}_{a_i}\cong \mathrm{M}_{x_0'}\otimes\bigotimes_{i=1}^\infty\mathrm{M}_{a_i'}$, then $S_0=S_0'$ and $S=S'$.
\end{proposition}
\begin{proof}
Passing to subsequences, we may assume by \Cref{lem1} that we have
\[
\begin{tikzcd}
\mathrm{M}_{x_0}\arrow{d}{\varepsilon_0}\arrow{r}{}&\mathrm{M}_{x_1}\arrow{d}{\varepsilon_1}\arrow{r}{}&\mathrm{M}_{x_2}\arrow{d}{\varepsilon_2}\arrow{r}{}&\cdots\\%
\mathrm{M}_{x_0'}\arrow{r}{}\arrow{ur}{\varepsilon_0'}&\mathrm{M}_{x_1'}\arrow{r}{}\arrow{ur}{\varepsilon_1'}&\mathrm{M}_{x_2'}\arrow{r}{}\arrow{ur}{\varepsilon_2'}&\cdots
\end{tikzcd}
\]
It follows that $\mathrm{M}_{x_0}\longrightarrow \mathrm{M}_{x_1}\longrightarrow\cdots$ and $\mathrm{M}_{x_0'}\longrightarrow \mathrm{M}_{x_1'}\longrightarrow\cdots$ are subsequences of $\mathrm{M}_{x_0}\longrightarrow \mathrm{M}_{x_0'}\longrightarrow \mathrm{M}_{x_1}\longrightarrow \mathrm{M}_{x_1'}\longrightarrow\cdots$. It is clear from the proof of \Cref{standard_form} that $S_0$ and $S$ do not change when we pass to a subsequence, hence the result.
\end{proof}

Now, we can give an explicit description of $K_0^\mathrm{gr}$ groups for a direct limit of matrix algebras with elementary gradings. 
\begin{definition}\label{def_K}
Let $\mathcal{A}=\mathrm{M}_{x_0}\otimes\bigotimes_{i=1}^\infty\mathrm{M}_{a_i}$ be in standard form. Consider $\mathbb{L}_S:=\prod_{j\in S}\mathbb{L}_j$ and the corresponding projection $\pi_S:\mathbb{Q}G\to\mathbb{L}_S$. Denote $b_k=\prod_{k=1}^i\bar{a}_k$ and observe that, since $\supp(a_i)=S$ for all $i$, each element $\pi_S(b_k)$ is invertible in $\mathbb{L}_S$. Hence, we may define the following subsets of $\mathbb{L}_S$:
\begin{align*}
K^+=K^+(\mathcal{A})&:=\bigcup_{i=1}^\infty\frac1{\pi_S(b_i)}\,\pi_S\left(\mathbb{Z}_{\ge0}G\right),\\
K=K(\mathcal{A})&:=\bigcup_{i=1}^\infty\frac1{\pi_S(b_i)}\,\pi_S\left(\mathbb{Z}G\right).
\end{align*}
\end{definition}

\begin{proposition}\label{K_0_ultraelementary}
Let $\mathcal{A}=\mathrm{M}_{x_0}\otimes\bigotimes_{i=1}^\infty\mathrm{M}_{a_i}$ be in standard form. Then \[K_0^\mathrm{gr}(\mathcal{A})\cong(K(\mathcal{A}),K^+(\mathcal{A}),\pi_S(\bar{x}_0)).\]
\end{proposition}
\begin{proof}
From \Cref{K_0_simple}, we know that $K_0^\mathrm{gr}(\mathrm{M}_{x_i})\cong(\mathbb{Z}G,\mathbb{Z}_{\ge0}G,\bar{x}_i)$. Since the functor $K_0^\mathrm{gr}$ commutes with direct limits, \Cref{K_0_morphism} tells us that $K_0^\mathrm{gr}(\mathcal{A})$ is isomorphic to the limit of the following sequence:
$$
(\mathbb{Z}G,\mathbb{Z}_{\ge0}G,\bar{x}_0)\over{\bar{a}_1}{\longrightarrow}(\mathbb{Z}G,\mathbb{Z}_{\ge0}G,\bar{x}_1)\over{\bar{a}_2}{\longrightarrow}(\mathbb{Z}G,\mathbb{Z}_{\ge0}G,\bar{x}_2)\over{\bar{a}_3}{\longrightarrow}\cdots
$$
Since $\pi_S$ is a ring homomorphism and $a_i\in\mathbb{Z}_{\geq 0}G$, we have 
\[
\frac1{\pi_S(b_i)}\,\pi_S(\mathbb{Z}_{\ge0}G)=\frac{\pi_S(a_{i+1})}{\pi_S(b_{i+1})}\,\pi_S(\mathbb{Z}_{\ge0}G)\subseteq\frac1{\pi_S(b_{i+1})}\,\pi_S(\mathbb{Z}_{\ge0}G),
\]
hence $K^+$ is an additive submonoid of $\mathbb{L}_S$, which is invariant under the action of $\mathbb{Z}_{\geq 0}G$ via $\pi_S$. Clearly, $K$ is the additive subgroup generated by $K^+$ and is invariant under $\mathbb{Z}G$. Now, we define $\psi_i:\mathbb{Z}G\to K$ by $\psi_i(z)=\frac1{\pi_S(b_i)}\pi_S(z)$. Then, $\psi_i$ is a morphism $K_0^\mathrm{gr}(\mathrm{M}_{x_i})\to(K,K^+,\pi_S(\bar{x}_0))$, i.e., a homomorphism of $\mathbb{Z}G$-modules preserving ordering and order-unit. In addition, we have commutative diagrams
$$
\begin{tikzcd}
(\mathbb{Z}G,\mathbb{Z}_{\ge0}G,\bar{x}_i)\arrow{r}{\bar{a}_{i+1}}\arrow{rd}{\psi_i}&(\mathbb{Z}G,\mathbb{Z}_{\ge0}G,\bar{x}_{i+1})\arrow{d}{\psi_{i+1}}\\&(K,K^+,\pi_S(\bar{x}_0))
\end{tikzcd}
$$
and, hence, obtain a morphism $\psi:K_0^\mathrm{gr}(\mathcal{A})\to (K,K^+,\pi_S(\bar{x}_0))$. Moreover, for each $i\in\mathbb{N}$, we have the commutative diagram
$$
\begin{tikzcd}
K_0^\mathrm{gr}(\mathrm{M}_{x_i})\arrow{r}{\varphi_i}\arrow{rd}{\psi_i}&K_0^\mathrm{gr}(\mathcal{A})\arrow{d}{\psi}\\&(K,K^+,\pi_S(\bar{x}_0)),
\end{tikzcd}
$$
where $\varphi_i:K_0^\mathrm{gr}(\mathrm{M}_{x_i})\to K_0^\mathrm{gr}(\mathcal{A})$ is the canonical map. By construction, it is clear that $\psi$ is surjective. Let $z\in\mathrm{Ker}\,\psi$. Then, there exists $z_i\in\mathbb{Z}G$ such that $\varphi_i(z_i)=z$. Hence, $\psi_i(z_i)=0$, so $\supp z_i\cap S=\emptyset$. Since the limit is in standard form and $\supp(\bar{a})=\supp(a)$ by \Cref{rem:conjugate}, this implies $\bar{a}_{i+1}z_i=0$. Thus, $z=\varphi_i(z_i)=\varphi_{i+1}(\bar{a}_{i+1}z_i)=0$. Therefore, $\psi$ is an isomorphism.
\end{proof}

\begin{remark}\label{SS_0}
An alternative way to prove \Cref{S_invariant} is to express $S$ and $S_0$ in terms of $K_0^\mathrm{gr}(\mathcal{A})$ for  $\mathcal{A}=\lim\limits_{\longrightarrow}\mathrm{M}_{x_i}(\mathbb{F})$. Using \Cref{K_0_ultraelementary}, one can easily check the following:
\begin{itemize}
\item For each $j\in\hat{G}/\Gamma$, pick an integer $n_j>0$ such that the multiple $e'_j:=n_je_j$ belongs to $\mathbb{Z}G$. Then, $S=\{j\in\hat{G}/\Gamma\mid e'_jK_0^\mathrm{gr}(\mathcal{A})\neq 0\}$.
\item $S_0=\{j\in\hat{G}/\Gamma\mid e'_j[\mathcal{A}]\ne0\}$.
\end{itemize}
\end{remark}

Now, we are ready to establish an isomorphism criterion for direct limit of matrix algebras with elementary gradings.
\begin{Thm}\label{thm1}
Let $\mathcal{A}=\mathrm{M}_{x_0}\otimes\bigotimes_{i=1}^\infty\mathrm{M}_{a_i}$ and $\mathcal{A}'=\mathrm{M}_{x'_0}\otimes\bigotimes_{i=1}^\infty\mathrm{M}_{a'_i}$ be in standard form. Then $\mathcal{A}\cong\mathcal{A}'$ if and only if $S=S'$ and there exist elements $b,b'\in\mathbb{Z}_{\ge0}G$ with support $S$ satisfying $b\bar{x}_0=b'\bar{x}_0'$ and $b K^+(\mathcal{A})=b'K^+(\mathcal{A}')$.
\end{Thm}
\begin{proof}
The triple $(K,K^+,\pi_S(\bar{x}_0))$ does not change if we pass to a subsequence of $\mathrm{M}_{x_i}$ with the same initial point $\mathrm{M}_{x_0}$, but it changes to $(b_k K,b_k K^+,\pi_S(b_k\bar{x}_0))$ if we start at $\mathrm{M}_{x_k}$ instead of $\mathrm{M}_{x_0}$. If $\mathcal{A}\cong\mathcal{A}'$ then $S=S'$ by \Cref{S_invariant}. Also, by \Cref{lem1}, we may assume that we have
\[
\mathrm{M}_{x_0}\over{\varepsilon_0}{\longrightarrow}\mathrm{M}_{b_{k'}'x_0'}\over{\varepsilon_0'}{\longrightarrow}\mathrm{M}_{x_1}\over{\varepsilon_1}{\longrightarrow}\mathrm{M}_{x_1'}\over{\varepsilon_1'}{\longrightarrow}\cdots,
\]
where $x_0$ and $x_0'$ are the same as before and the rest of $x_i$, $x_i'$ are subsequences of the original sequences, which we renamed to simplify notation. Let $a$ be a label of $\varepsilon_0$. Then $a$ divides $a_1$ in $\mathbb{Z}_{\ge0}G$, so $\supp a\supseteq S$. Skipping one step to $\mathrm{M}_{x_1'}$, we may replace $a$ by $aa_1'$, so we may assume that $\supp a=S$. Repeating this, we get a limit in standard form. On the subsequence $\mathrm{M}_{x_i}$, 
we get $(K,K^+,\pi_S(\bar{x}_0))$, while on $\mathrm{M}_{x'_i}$, we get $(b'_{k'} K',b'_{k'} K^{\prime+},\pi_S(b'_{k'}\bar{x}'_0))$, which then should be equal to $(\bar{a} K,\bar{a}K^+,\pi_S(\bar{a}\bar{x}_0))$. Since $\supp(\bar{a}\bar{x}_0)$ and $\supp(b'_{k'}\bar{x}'_0)$ are contained in $S$, we get $\bar{a}\bar{x}_0=b'_{k'}\bar{x}_0'$.

For the converse, by hypothesis, $\pi_S(b)$ and $\pi_S(b')$ are invertible in $\mathbb{L}_S$, hence the multiplication by $\frac{\pi_S(b)}{\pi_S(b')}$ is an automorphism of the additive group of $\mathbb{L}_S$, and it maps $(K,K^+,\pi_S(\bar{x}_0))$ onto $(K',K^{\prime+},\pi_S(\bar{x}'_0))$. In view of \Cref{K_0_ultraelementary}, this means $K_0^\mathrm{gr}(\mathcal{A})\cong K_0^\mathrm{gr}(\mathcal{A}')$, so the result follows from \cite[Theorem 5.2.4]{Haz}.
\end{proof}

\begin{corollary}
Let $\mathcal{A}=\mathrm{M}_{x_0}\otimes\bigotimes_{i=1}^\infty\mathrm{M}_{a_i}$ with $\supp(x_0)=\supp(a_i)=S$ and let $c\in\mathbb{Z}_{\ge0}G$. Then $\mathcal{A}\otimes\mathrm{M}_c\cong\mathcal{A}$ if and only if $\bar{c} K^+=K^+$.
\end{corollary}
\begin{proof}
In this case, using the notation of \Cref{thm1}, we have $x_0'=cx_0$ and $a'_i=a_i$, so the condition becomes $b\bar{x}_0=b'\bar{c}\bar{x}_0$ and $bK^+=b'K^+$. Under the hypothesis on $\supp(x_0)$, this becomes $b=b'\bar{c}$ and $\bar{c}K^+=K^+$.
\end{proof}

\section{Main results\label{s:main_results}}
As in the previous section, let $G$ be a finite abelian group. We will also assume that the ground field $\mathbb{F}$ is algebraically closed. Consider a sequence of unital embeddings of finite-dimensional $G$-graded-simple associative algebras:
\[
\mathcal{A}_0\over{\iota_1}{\longrightarrow}\mathcal{A}_1\over{\iota_2}{\longrightarrow}\mathcal{A}_2\over{\iota_3}{\longrightarrow}\cdots
\]
By the graded Wedderburn Theorem, we have $\mathcal{A}_i\cong\mathrm{M}_{x_i}(\mathcal{D}_i)$, for some $x_i\in\mathbb{Z}_{\geq 0}G$ and graded-division algebra $\mathcal{D}_i$. Since $\mathbb{F}$ is algebraically closed and $G$ is finite, there exists a finite number of isomorphism classes of $G$-graded-division algebras over $\mathbb{F}$. Hence, passing to a subsequence, we may assume that  $\mathcal{A}_i\cong\mathrm{M}_{x_i}(\mathcal{D})$ with a constant $\mathcal{D}$. Let $T=\Supp\mathcal{D}$.

From now on, we assume that each $\mathcal{A}_i$ is central simple as an algebra. Then, by \Cref{gr_NS}, each embedding $\iota_i:\mathcal{A}_{i-1}\to\mathcal{A}_i$ is determined, up to equivalence, by the isomorphism class of the centralizer $\mathcal{C}_i$ of $\iota_i(\mathcal{A}_{i-1})$ in $\mathcal{A}_i$. By the Double Centralizer Theorem, we have $\mathcal{A}_i\cong\mathcal{A}_{i-1}\otimes\mathcal{C}_i$, so $[\mathcal{A}_i]=[\mathcal{A}_{i-1}][\mathcal{C}_i]$ in $\mathrm{Br}_G(\mathbb{F})$. Since $[\mathcal{A}_{i-1}]=[\mathcal{A}_i]=[\mathcal{D}]$, we see that $[\mathcal{C}_i]=[\mathbb{F}]$, i.e., $\mathcal{C}_i$ has an elementary grading: $\mathcal{C}_i\cong\mathrm{M}_{a_i}$ for some $a_i\in\mathbb{Z}_{\geq 0}G$. Since $\mathcal{A}_{i-1}\cong\mathrm{M}_{x_{i-1}}\otimes\mathcal{D}$, we get $\mathcal{A}_i\cong\mathrm{M}_{x_{i-1}}\otimes\mathcal{C}_i\otimes\mathcal{D}$. It follows that $\iota_i$ is equivalent to the tensor product of the embedding $\mathrm{M}_{x_{i-1}}\to\mathrm{M}_{x_{i-1}}\otimes\mathcal{C}_i$ with $\mathrm{id}_\mathcal{D}$, and we may assume $x_i=x_{i-1}a_i$. Hence, by \Cref{lm:equiv_embeddings},
\[
\lim\limits_{\longrightarrow}\mathcal{A}_i\cong\left(\lim\limits_{\longrightarrow}\mathrm{M}_{x_i}\right)\otimes\mathcal{D},
\]
where the embedding $\mathrm{M}_{x_{i-1}}\to\mathrm{M}_{x_{i}}$ has label $a_i$.
Thus, in order to classify direct limits of matrix algebras with $G$-gradings, it is sufficient to understand the tensor product of a graded-division algebra $\mathcal{D}$ and a direct limit of matrix algebras with elementary gradings.
The following result shows that $K_0^\mathrm{gr}$ carries limited information for such a tensor product (see also \Cref{example3} below).

\begin{proposition}\label{K_0_division}
If $\mathcal{D}$ is a $G$-graded-division algebra with support $T$ and $\mathcal{A}$ is a direct limit of matrix algebras over $\mathbb{F}$ with elementary gradings, then 
\[
K_0^G(\mathcal{A}\otimes\mathcal{D})\cong K_0^{G/T}({}^\alpha\mathcal{A}),
\]
where $\alpha:G\to G/T$ is the quotient map. 
\end{proposition}

\begin{proof}
Let $\mathcal{V}$ be the graded $\mathrm{M}_x$-module as in the proof of \Cref{K_0_simple}. Then $\mathcal{V}\otimes\mathcal{D}$ is a graded-simple $\mathrm{M}_x(\mathcal{D})$-module. It is unique up to isomorphism and shift, and $(\mathcal{V}\otimes\mathcal{D})^{[g]}\cong(\mathcal{V}\otimes\mathcal{D})^{[h]}$ if and only if $g\equiv h\pmod{T}$. As in \Cref{K_0_simple}, we then obtain that $K_0^G(\mathrm{M}_x(\mathcal D))\cong(\mathbb{Z}(G/T),\mathbb{Z}_{\geq 0}(G/T),\alpha(\bar{x}))$, so we have $K_0^G(\mathrm{M}_x\otimes\mathcal D)\cong K_0^{G/T}(\mathrm{M}_{\alpha(x)})\cong K_0^{G/T}({}^\alpha\mathrm{M}_x)$. Now, if $\mathcal{A}=\lim\limits_{\longrightarrow}\mathrm{M}_{x_i}$, then 
\begin{equation*}
\begin{split}
K_0^G(\mathcal{A}\otimes\mathcal{D})&\cong K_0^G\bigl(\lim\limits_{\longrightarrow}(\mathrm{M}_{x_i}\otimes\mathcal{D})\bigr)\cong\lim\limits_{\longrightarrow}K_0^G\left(\mathrm{M}_{x_i}\otimes\mathcal{D}\right)\\
&\cong\lim\limits_{\longrightarrow}K_0^{G/T}({}^\alpha\mathrm{M}_{x_i})
\cong K_0^{G/T}\bigl(\lim\limits_{\longrightarrow}{}^\alpha\mathrm{M}_{x_i}\bigr)\cong K_0^{G/T}({}^\alpha\mathcal{A}),
\end{split}
\end{equation*}
as claimed.
\end{proof}

The next result characterizes when a graded-division algebra can be ``absorbed" by the direct limit of matrix algebras with elementary gradings.
\begin{Thm}\label{thm}
Let $G$ be a finite abelian group and $\mathbb{F}$ an algebraically closed field. If $\mathcal{A}$ is a direct limit of matrix algebras over $\mathbb{F}$ endowed with elementary gradings and $\mathcal{D}$ is a $G$-graded-division algebra with support $T$ that is central simple as an algebra, then the following statements are equivalent:
\begin{enumerate}
\renewcommand{\labelenumi}{(\roman{enumi})}
\item $\mathcal{A}\otimes\mathcal{D}$ is a direct limit of matrix algebras with elementary gradings,
\item $x_T\cdot K_0^\mathrm{gr}(\mathcal{A})^+=K_0^\mathrm{gr}(\mathcal{A})^+$, where $x_T=\sum_{t\in T}t$,
\item $|T|\cdot K_0^\mathrm{gr}(\mathcal{A})^+=K_0^\mathrm{gr}(\mathcal{A})^+$ and $\supp(x_T)\supseteq S$, where $S$ is as in \Cref{SS_0}.
\end{enumerate}
Moreover, if this is the case, then $\mathcal{A}\otimes\mathcal{D}\cong\mathcal{A}$.
\end{Thm}
\begin{proof}
(i) $\Rightarrow$ (ii): Let $\mathcal{A}=\lim\limits_{\longrightarrow}\mathrm{M}_{x_i}$ and 
suppose $\mathcal{A}\otimes\mathcal{D}\cong\lim\limits_{\longrightarrow}\mathrm{M}_{x_i'}$. Passing to subsequences, we may assume by \Cref{lem1} that we have
\[
\mathrm{M}_{x_0}\otimes\mathcal{D}\over{\varepsilon_0}{\longrightarrow}\mathrm{M}_{x_0'}\over{\varepsilon_0'}{\longrightarrow}\mathrm{M}_{x_1}\otimes\mathcal{D}\over{\varepsilon_1}{\longrightarrow}\cdots.
\]
Let $\mathcal{C}_i=\mathrm{Cent}_{\mathrm{M}_{x_i'}}(\varepsilon_i(\mathrm{M}_{x_i}\otimes\mathcal{D}))$, so $\mathrm{M}_{x_i'}\cong\mathrm{M}_{x_i}\otimes\mathcal{D}\otimes\mathcal{C}_i$. From the graded Brauer group, we get $\mathcal{C}_i\cong\mathrm{M}_{c_i}(\mathcal{D}^\mathrm{op})$, so we can factor $\varepsilon_i$ as
\[
\mathrm{M}_{x_i}\otimes\mathcal{D}\longrightarrow\mathrm{M}_{x_i}\otimes\mathcal{D}\otimes\mathcal{D}^\mathrm{op}\over{\varepsilon_i''}{\longrightarrow}\mathrm{M}_{x_i'}.
\]
Let $\mathcal{C}_i'=\mathrm{Cent}_{\mathrm{M}_{x_{i+1}}\otimes\mathcal{D}}(\varepsilon_i'\varepsilon_i''(\mathrm{M}_{x_i}\otimes\mathcal{D}\otimes\mathcal{D}^\mathrm{op}))$, so $\mathrm{M}_{x_{i+1}}\otimes\mathcal{D}\cong\mathrm{M}_{x_i}\otimes\mathcal{D}\otimes\mathcal{D}^\mathrm{op}\otimes\mathcal{C}_i'$. From the graded Brauer group, $\mathcal{C}_i'\cong\mathrm{M}_{c_i'}(\mathcal{D}')$ where $\mathcal{D}'\cong\mathcal{D}$.

Skipping the term $\mathrm{M}_{x_i'}$ and restricting to the centralizers of $\mathcal{D}$, we obtain
\[
\cdots\longrightarrow\mathrm{M}_{x_i}\longrightarrow\mathrm{M}_{x_i}\otimes\mathcal{D}^\mathrm{op}\otimes\mathcal{D}'\under{c_i'}{\longrightarrow}\mathrm{M}_{x_{i+1}}
%\longrightarrow\mathrm{M}_{x_{i+1}}\otimes\mathcal{D}^\mathrm{op}\otimes\mathcal{D}'
\longrightarrow\cdots
\]
Recall from \Cref{DDop} that $\mathcal{D}^\mathrm{op}\otimes\mathcal{D}'\cong\mathrm{M}_{x_T}$, so our sequence becomes 
\[
\mathrm{M}_{x_0}\under{x_T}{\longrightarrow}\mathrm{M}_{x_0x_T}\longrightarrow\mathrm{M}_{x_1}\under{x_T}{\longrightarrow}\mathrm{M}_{x_1x_T}\longrightarrow\cdots,
\]
where the composition $\mathrm{M}_{x_i}\to\mathrm{M}_{x_{i+1}}$ is the original embedding. It follows that  $x_T K_0^\mathrm{gr}(\mathcal{A})^+=K_0^\mathrm{gr}(\mathcal{A})^+$. Indeed, since $K_0^\mathrm{gr}(\mathcal{A})$ is the direct limit of $K_0^\mathrm{gr}$ of the above sequence, for any element $z\in K_0^\mathrm{gr}(\mathcal{A})^+$ there exists $i$ such that $z$ is the image of some $y\in\mathbb{Z}_{\ge 0}G$ under the homomorphism $K_0^\mathrm{gr}(\mathrm{M}_{x_i})\to K_0^\mathrm{gr}(\mathcal{A})$. But then $z$ is also the image of $x_Ty$ under the homomorphism $K_0^\mathrm{gr}(\mathrm{M}_{x_ix_T})\to K_0^\mathrm{gr}(\mathcal{A})$.

(ii) $\Rightarrow$ (i): Write $\mathcal{A}=\mathrm{M}_{x_0}\otimes\bigotimes_{i=1}^\infty\mathrm{M}_{a_i}$ in standard form (see \Cref{Def:standard_form} and \Cref{standard_form}) and let $(K,K^+,\pi_S(\bar{x}_0))$ be the realization of $K_0^\mathrm{gr}(\mathcal{A})$ from \Cref{K_0_ultraelementary}. Recall that
\[
K^+=\bigcup_{i=1}^\infty\frac{1}{\pi_S(\bar{a}_1\cdots\bar{a}_i)}\,\pi_S(\mathbb{Z}_{\ge0}G).
\] 
Now, suppose $x_T K^+=K^+$. Since $1\in K^+$, we have
\[
1=\frac{\pi_S(x_T)}{\pi_S(\bar{a}_1\cdots\bar{a}_{i_1})}\,\pi_S(c_1),
\]
for some $i_1\in\mathbb{N}$ and $c_1\in\mathbb{Z}_{\ge0}G$. Then $\pi_S(\bar{a}_1\cdots\bar{a}_{i_1})=\pi_S(x_Tc_1)$, which implies $\supp(c_1)\supseteq S$ because $\supp(\bar{a}_i)=\supp(a_i)=S$. Multiplying both sides by $\bar{a}_{i_1+1}$ and replacing $c_1$ with $c_1\bar{a}_{i_1+1}$ and $i_1$ with $i_1+1$, we may assume $\supp (c_1)=S$. Then $\pi_S(\bar{a}_1\cdots\bar{a}_{i_1})=\pi_S(x_Tc_1)$ implies $\bar{a}_1\cdots\bar{a}_{i_1}=x_Tc_1$. Hence, we can factor the embedding $\mathrm{M}_{x_0}\to \mathrm{M}_{x_{i_1}}$ as 
\begin{equation}\label{eq:0i1}
\mathrm{M}_{x_0}\longrightarrow\mathrm{M}_{x_0}\otimes\mathcal{D}\longrightarrow\mathrm{M}_{x_0}\otimes\mathcal{D}\otimes\mathcal{D}^\mathrm{op}\under{\bar{c}_1}{\longrightarrow}\mathrm{M}_{x_{i_1}}.
\end{equation}
By the same argument, there exists $i_2>i_1$ such that $\bar{a}_{i_1+1}\cdots \bar{a}_{i_2}=x_Tc_2$, so we can factor the embedding $\mathrm{M}_{x_{i_1}}\to \mathrm{M}_{x_{i_2}}$ as 
\begin{equation}\label{eq:i1i2}
\mathrm{M}_{x_{i_1}}\longrightarrow\mathrm{M}_{x_{i_1}}\otimes\mathcal{D}\longrightarrow\mathrm{M}_{x_{i_1}}\otimes\mathcal{D}\otimes\mathcal{D}^{\mathrm{op}}\under{\bar{c}_2}{\longrightarrow}\mathrm{M}_{x_{i_2}},
\end{equation}
From \Cref{eq:0i1}, we have $\mathrm{M}_{x_{i_1}}\cong\mathrm{M}_{x_0}\otimes\mathcal{D}\otimes\mathcal{C}$ where $\mathcal{C}\cong\mathrm{M}_{\bar{c}_1}\otimes\mathcal{D}^\mathrm{op}$.
Putting this into \Cref{eq:i1i2} and joining with \Cref{eq:0i1}, we get 
\[
\mathrm{M}_{x_0}{\longrightarrow}\mathrm{M}_{x_0}\otimes\mathcal{D}{\longrightarrow}\mathrm{M}_{x_{i_1}}\cong \mathrm{M}_{x_0}\otimes\mathcal{D}\otimes\mathcal{C}{\longrightarrow}\mathrm{M}_{x_0}\otimes\mathcal{D}\otimes\mathcal{C}\otimes\mathcal{D}\cong\mathrm{M}_{x_{i_1}}\otimes\mathcal{D}{\longrightarrow}\mathrm{M}_{x_{i_2}},
\]
where the compositions $\mathrm{M}_{x_0}\to\mathrm{M}_{x_{i_1}}$ and $\mathrm{M}_{x_{i_1}}\to\mathrm{M}_{x_{i_2}}$ are still the original embeddings and $\mathrm{M}_{x_0}\otimes\mathcal{D}\otimes\mathcal{C}\to\mathrm{M}_{x_0}\otimes\mathcal{D}\otimes\mathcal{C}\otimes\mathcal{D}$ is given by $b\otimes d\otimes c\mapsto b\otimes d\otimes c\otimes 1$. Replacing this with $b\otimes d\otimes c\mapsto b\otimes 1\otimes c\otimes d$ and pre-composing the isomorphism $\mathrm{M}_{x_0}\otimes\mathcal{D}\otimes\mathcal{C}\otimes\mathcal{D}\cong\mathrm{M}_{x_{i_1}}\otimes\mathcal{D}$ with the swap of the two copies of $\mathcal{D}$, we get 
\[
\mathrm{M}_{x_0}\over{\varepsilon_0}{\longrightarrow}\mathrm{M}_{x_{i_0}}\otimes\mathcal{D}\over{\varepsilon'_0}{\longrightarrow}\mathrm{M}_{x_{i_1}}\over{\varepsilon_1}{\longrightarrow}\mathrm{M}_{x_{i_1}}\otimes\mathcal{D}\over{\varepsilon'_1}{\longrightarrow}\mathrm{M}_{x_{i_2}},
\]
where $\varepsilon'_0\varepsilon_0$ and $\varepsilon'_1\varepsilon_1$ are the original embeddings and $\varepsilon_1\varepsilon'_0$ is the original embedding tensored with $\mathrm{id}_\mathcal{D}$.
%Then
%\[
%\mathrm{M}_{x_0}{\longrightarrow}\mathrm{M}_{x_0}\otimes\mathcal{D}\under{\bar{c}_1}{\longrightarrow}\mathrm{M}_{x_{i_1}}{\longrightarrow}\mathrm{M}_{x_{i_1}}\otimes\mathcal{D}'\longrightarrow\mathrm{M}_{x_{i_1}}\otimes\mathcal{D}'\otimes\mathcal{D}^{\prime\mathrm{op}}\under{\bar{c}_2}{\longrightarrow}\mathrm{M}_{x_{i_2}}.
%\]
%Replacing $\mathrm{M}_{x_{i_1}}\cong\mathrm{M}_{x_0}\otimes\mathcal{D}\otimes\mathcal{C}_1$ by its isomorphic copy $\mathrm{M}_{x_0}\otimes\mathcal{C}_1\otimes\mathcal{D}'$ in $\mathrm{M}_{x_{i_1}}\otimes\mathcal{D}'\cong\mathrm{M}_{x_0}\otimes\mathcal{D}\otimes\mathcal{C}_1\otimes\mathcal{D}'$, we can change the middle embeddings:
%\[
%\mathrm{M}_{x_0}{\longrightarrow}\mathrm{M}_{x_0}\otimes\mathcal{D}\under{c_1}{\longrightarrow}\mathrm{M}_{x_{i_1}}{\longrightarrow}\mathrm{M}_{x_{i_1}}\otimes\mathcal{D}\longrightarrow\mathrm{M}_{x_{i_1}}\otimes\mathcal{D}\otimes\mathcal{D}^{\prime\mathrm{op}}\longrightarrow\mathrm{M}_{x_{i_2}}
%\]
%so that the compositions $\mathrm{M}_{x_0}\longrightarrow\mathrm{M}_{x_{i_1}}$ and $\mathrm{M}_{x_{i_1}}\longrightarrow\mathrm{M}_{x_{i_2}}$ are still the original embeddings and the composition $\mathrm{M}_{x_0}(\mathcal{D})\longrightarrow\mathrm{M}_{x_{i_1}}(\mathcal{D})$ is the original embedding tensored with $\mathrm{id}_\mathcal{D}$.

Continuing this process and applying \Cref{lem1}, we conclude that $\mathcal{A}\otimes\mathcal{D}\cong\mathcal{A}$. Thus, conditions (i) and (ii) are equivalent to each other, and the last assertion of the theorem holds.

(ii) $\Rightarrow$ (iii): If $x_T K^+=K^+$, then it is clear that $S\supseteq\supp x_T$. 
Since $x_T^2=|T|x_T$, we also obtain $|T|K^+=|T|(x_TK^+)=x_T^2K^+=K^+$.

(iii) $\Rightarrow$ (ii): Note that the element $\frac{1}{|T|}x_T$ is the sum of the primitive idempotents $e_j\in\mathbb{Q}G$ over all $j\in\supp(x_T)$. Since $S\subseteq\supp(x_T)$, it follows that $\frac1{|T|}x_T\cdot K^+=K^+$. But $|T| K^+=K^+$ implies $\frac1{|T|}K^+=K^+$, so we get $x_TK^+=K^+$.
\end{proof}

\begin{remark}\label{rem:supp_xT}
For any subgroup $T$ of $G$, we have $\supp(x_T)=T^\perp/\mathrm{Gal}(\overline{\mathbb{Q}}/\mathbb{Q})$, where $T^\perp$ is the subgroup of $\hat{G}$ defined by $T^\perp=\{\chi\in\hat{G}\mid\chi(t)=1,\forall t\in T\}$.
%$$
%\supp(x_T)=\{j\in\hat{G}/\Gamma\mid\chi(t)=1,\forall t\in T,\,\chi\in j\}.
%$$
In particular, condition $\supp(x_T)\supseteq S$ in \Cref{thm} is equivalent to $T\subseteq S^\perp$, where $S^\perp$ is the subgroup of $G$ defined by $S^\perp=\{g\in G\mid\chi(g)=1,\forall\chi\in j,\,j\in S\}$.
\end{remark}

The next example shows that the $K_0^\mathrm{gr}$ functor is not enough to distinguish direct limit of central simple algebras with $G$-gradings.
\begin{example}\label{example3}
Let $G=T=C_2\times C_2$ and let $\mathcal{D}$ be the graded-division algebra $\mathcal{Q}$ with support $T$ defined in \Cref{ex1}. Let $\mathcal{A}=\mathrm{M}_2\otimes\bigotimes_{i=1}^\infty\mathrm{M}_2$ and $\mathcal{A}'=\mathrm{M}_{x_T}\otimes\bigotimes_{i=1}^\infty\mathrm{M}_{x_T}$, and denote $R:=\mathbb{Z}\left[\frac12\right]$. By \Cref{K_0_ultraelementary}, we have $K_0^\mathrm{gr}(\mathcal{A})\cong(RG,R_{\ge0}G,2)$ and $K_0^\mathrm{gr}(\mathcal{A}')\cong(x_TRG,x_TR_{\ge0}G,x_T)\cong(R(G/T),R_{\ge 0}(G/T),|T|)=(R,R_{\ge 0},4)$, so $\mathcal{A}\not\cong\mathcal{A}'$. On the other hand, \Cref{K_0_division} tells us that, as $\mathbb{Z}T$-module, $K_0^\mathrm{gr}(\mathcal{A}\otimes\mathcal{D})\cong(R,R_{\ge0},2)$, where each element of $T$ acts trivially. It follows that $\mathcal{A}\otimes\mathcal{D}\not\cong\mathcal{A}'$, for otherwise, \Cref{thm} would tell us that $\mathcal{A}\cong\mathcal{A}\otimes\mathcal{D}\cong\mathcal{A}'$, a contradiction. Hence, $K_0^\mathrm{gr}(\mathcal{A}\otimes\mathcal{D})\cong K_0^\mathrm{gr}(\mathcal{A}')$, but $\mathcal{A}\otimes\mathcal{D}\not\cong\mathcal{A}'$.
\end{example}

Finally, we are ready to give an isomorphism criterion:
\begin{Thm}\label{last_prop}
Let $G$ be a finite abelian group and $\mathbb{F}$ an algebraically closed field. Let $\mathcal{A}$ and $\mathcal{A}'$ be direct limits of matrix algebras over $\mathbb{F}$ endowed with elementary $G$-gradings. Let $\mathcal{D}$ and $\mathcal{D}'$ be $G$-graded-division algebras that are central simple as algebras, and write $\mathcal{D}\otimes\mathcal{D}^{\prime\mathrm{op}}\cong\mathrm{M}_y(\mathcal{E})$ where $y\in\mathbb{Z}_{\ge 0}G$ and $\mathcal{E}$ is a graded-division algebra. Denote $T=\Supp\mathcal{D}$, $T'=\Supp\mathcal{D}'$, $E=\Supp\mathcal{E}$ and $x_H=\sum_{h\in H}h$ for any $H\leq G$.
Then, $\mathcal{A}\otimes\mathcal{D}\cong\mathcal{A}'\otimes\mathcal{D}'$ if and only if (i) $x_E\cdot K_0^\mathrm{gr}(\mathcal{A})^+=K_0^\mathrm{gr}(\mathcal{A})^+$ and (ii) $\mathrm{M}_y\otimes\mathcal{A}\cong\mathrm{M}_{x_{T'}}\otimes\mathcal{A}'$.
\end{Thm}
\begin{proof}
Assume that $\mathcal{A}\otimes\mathcal{D}\cong\mathcal{A}'\otimes\mathcal{D}'$. Then, we have
\[
\underbrace{\mathcal{A}\otimes\mathcal{D}\otimes\mathcal{D}^{\prime\mathrm{op}}}_{\cong\mathcal{A}\otimes\mathrm{M}_y\otimes\mathcal{E}}\cong\mathcal{A}'\otimes\mathcal{D}'\otimes\mathcal{D}^{\prime\mathrm{op}}.
\]
Since $\mathcal{A}'\otimes\mathcal{D}'\otimes\mathcal{D}^{\prime\mathrm{op}}$ is a direct limit of matrix algebras with elementary gradings, by \Cref{thm}, we get $x_E\cdot K_0^\mathrm{gr}(\mathrm{M}_y\otimes\mathcal{A})^+=K_0^\mathrm{gr}(\mathrm{M}_y\otimes\mathcal{A})^+$. But $K_0^\mathrm{gr}(\mathrm{M}_y\otimes\mathcal{A})$ and $K_0^\mathrm{gr}(\mathcal{A})$ are isomorphic as ordered $\mathbb{Z}G$-modules, though possibly have different order-units (see \cite[Theorem 5.2.5]{Haz}, or \Cref{K_0_ultraelementary} and \Cref{rem:tensor_with_elem}), hence we get (i). %$x_EK_0^\mathrm{gr}(\mathcal{A})^+=K_0^\mathrm{gr}(\mathcal{A})^+$. 
%Since $\mathcal{D}'\otimes\mathcal{D}^\mathrm{op}\cong\mathrm{M}_{\bar{y}}(\mathcal{E}^\mathrm{op})$ and $\Supp\mathcal{E}^\mathrm{op}=E$, we get the second part of (i) by replacing $\mathcal{A}$ by $\mathcal{A}'$ in the above argument. 
To prove (ii), note that $\mathcal{A}\cong\mathcal{E}\otimes\mathcal{A}$ by \Cref{thm}, hence
\[
\mathrm{M}_y\otimes\mathcal{A}\cong\mathrm{M}_y(\mathcal{E})\otimes\mathcal{A}\cong\mathcal{D}^{\prime\mathrm{op}}\otimes\mathcal{D}\otimes\mathcal{A}\cong\mathcal{D}^{\prime\mathrm{op}}\otimes\mathcal{D}'\otimes\mathcal{A}'\cong\mathrm{M}_{x_{T'}}\otimes\mathcal{A}'.
\]
Conversely, if (i) and (ii) hold, then $\mathcal{A}\cong\mathcal{E}\otimes\mathcal{A}$ by \Cref{thm} and, hence,
\[
\mathcal{D}\otimes\mathcal{D}^{\prime\mathrm{op}}\otimes\mathcal{A}\cong\mathcal{E}\otimes\mathrm{M}_y\otimes\mathcal{A}\cong\mathrm{M}_y\otimes\mathcal{A}\cong\mathrm{M}_{x_{T'}}\otimes\mathcal{A}'.
\]
Tensoring with $\mathcal{D}'$, we get 
\[
\mathrm{M}_{x_{T'}}\otimes\mathcal{D}\otimes\mathcal{A}\cong\mathrm{M}_{x_{T'}}\otimes\mathcal{D}'\otimes\mathcal{A}'\cong\mathrm{M}_{|T'|}\otimes\mathcal{D}'\otimes\mathcal{A}'.
\]
On the other hand, since $[\mathcal{E}]=[\mathcal{D}][\mathcal{D}']^{-1}$ in $\mathrm{Br}_G(\mathbb{F})$, we have $\mathcal{D}\otimes\mathcal{E}^\mathrm{op}\cong\mathrm{M}_z(\mathcal{D}')$, for some $z\in\mathbb{Z}_{\ge0}G$. Hence,
\[
\begin{split}
\mathrm{M}_{x_{T'}}\otimes\mathcal{D}\otimes\mathcal{A}&\cong\mathrm{M}_{x_{T'}}\otimes\mathcal{D}\otimes\mathcal{E}^\mathrm{op}\otimes\mathcal{A}\cong\mathrm{M}_{x_{T'}}\otimes\mathcal{D}'\otimes\mathrm{M}_z\otimes\mathcal{A}\\
&\cong\mathrm{M}_{x_{T'}}(\mathcal{D}')\otimes\mathrm{M}_z\otimes\mathcal{A}\cong\mathrm{M}_{|T'|}(\mathcal{D}')\otimes\mathrm{M}_z\otimes\mathcal{A}\\
&\cong\mathrm{M}_{|T'|}\otimes\mathcal{D}'\otimes\mathrm{M}_z\otimes\mathcal{A}\cong\mathrm{M}_{|T'|}\otimes\mathcal{D}\otimes\mathcal{E}^\mathrm{op}\otimes\mathcal{A}\cong\mathrm{M}_{|T'|}\otimes\mathcal{D}\otimes\mathcal{A}.
\end{split}
\]
We have shown that
$
\mathrm{M}_{|T'|}\otimes\mathcal{D}\otimes\mathcal{A}\cong\mathrm{M}_{|T'|}\otimes\mathcal{D}'\otimes\mathcal{A}',
$
which implies $\mathcal{D}\otimes\mathcal{A}\cong\mathcal{D}'\otimes\mathcal{A}'$ by \Cref{remove_matrix}.
\end{proof}

\begin{remark}\label{rem:tensor_with_elem}
It is worth mentioning that not only condition (i), but also (ii) can be checked using the realization of $K_0^\mathrm{gr}$ given in \Cref{K_0_ultraelementary}. Indeed, write $\mathcal{A}\cong\mathrm{M}_{x_0}\otimes\bigotimes_{i=1}^\infty\mathrm{M}_{a_i}$ in standard form. Then, since $\mathrm{supp}(x_0)\subseteq S$ implies $\mathrm{supp}(x_0y)\subseteq S$, we have $\mathrm{M}_y\otimes\mathcal{A}\cong\mathrm{M}_{yx_0}\otimes\bigotimes_{i=1}^\infty\mathrm{M}_{a_i}$ in standard form.
%Hence, the theorem reduces the isomorphism problem to conditions that can be effectively checked.
\end{remark}

\begin{corollary}
$\mathcal{A}\otimes\mathcal{D}\cong\mathcal{A}'\otimes\mathcal{D}$ if and only if $\mathcal{A}\otimes\mathrm{M}_{x_T}\cong\mathcal{A}'\otimes\mathrm{M}_{x_T}$.
\end{corollary}

Condition (ii) in \Cref{last_prop} implies that $K_0^\mathrm{gr}(\mathcal{A})$ and $K_0^\mathrm{gr}(\mathcal{A}')$ are isomorphic as ordered $\mathbb{Z}G$-modules (which corresponds to $\mathcal{A}$ and $\mathcal{A}'$ being Morita equivalent, see \cite[Theorem 5.2.5]{Haz}). To obtain representatives of isomorphism classes of $\mathcal{A}\otimes\mathcal{D}$ for a fixed isomorphism class $(K,K^+)$ of $K_0^\mathrm{gr}(\mathcal{A})$, we can use representatives of $\mathrm{Br}_G(\mathbb{F})$ under the following equivalence relation: $[\mathcal{D}]\sim[\mathcal{D}']$ if and only if condition (i) holds, i.e., $x_E\cdot K^+=K^+$ where $E=\Supp\mathcal{E}$ and $[\mathcal{E}]=[\mathcal{D}][\mathcal{D}']^{-1}$ (equivalently, $|E|\cdot K^+=K^+$ and $E\subseteq S^\perp$, see \Cref{rem:supp_xT}). Indeed, if $[\mathcal{D}]\sim[\mathcal{D}']$ then we can rewrite $\mathcal{A}\otimes\mathcal{D}'\cong\mathcal{A}\otimes\mathcal{E}\otimes\mathcal{D}'\cong\mathrm{M}_y(\mathcal{A})\otimes\mathcal{D}$. Finally, for a fixed $\mathcal{D}$, we have $\mathcal{A}\otimes\mathcal{D}\cong\mathcal{A}'\otimes\mathcal{D}$ if and only if $\mathcal{A}\otimes\mathrm{M}_{x_T}\cong\mathcal{A}'\otimes\mathrm{M}_{x_T}$, and the latter can be checked using \Cref{thm1}.

\end{document}